\documentclass[12pt]{amsart}

\usepackage{geometry}
\usepackage{amsmath,amssymb,amsthm}%
\usepackage{indentfirst,color,graphicx,tikz}%
\usepackage[colorlinks,citecolor=red,linkcolor=blue,urlcolor=cyan]{hyperref}%
\usetikzlibrary{arrows.meta}%

\theoremstyle{plain}%
\newtheorem{mainthm}{Theorem}%
\newtheorem{theorem}{Theorem}[section]%
\newtheorem{lemma}[theorem]{Lemma}%
\theoremstyle{remark}%
\newtheorem{remark}[theorem]{Remark}%
\newtheorem{example}[theorem]{Example}%
\numberwithin{equation}{section}%

\newcommand{\BB}{\mathbb{B}}%
\newcommand{\CC}{\mathbb{C}}%
\newcommand{\RR}{\mathbb{R}}%
\newcommand{\Sp}{\mathbb{S}}%
\newcommand{\ZZ}{\mathbb{Z}}%
\newcommand{\calO}{\mathcal{O}}%
\newcommand{\bd}{\mathbf{d}}%
\newcommand{\pd}{\partial}%
\newcommand{\eps}{\varepsilon}%
\renewcommand{\geq}{\geqslant}%
\renewcommand{\leq}{\leqslant}%
\renewcommand{\tilde}[1]{\widetilde{#1}}%
\renewcommand{\Re}{\operatorname{Re}}%

\begin{document}

\title{On the Converse of Pr\'{e}kopa's Theorem and Berndtsson's Theorem}

\author{Wang XU}
\address{School of Mathematics, Sun Yat-sen University, Guangzhou 510275, China}
\email{xuwang@amss.ac.cn; xuwang6@mail.sysu.edu.cn}

\author{Hui YANG}
\address{School of Mathematical Sciences, Peking University, Beijing 100871, China}
\email{yanghui@amss.ac.cn}

\thanks{The first author is supported by National Key R\&D Program of China (No. 2024YFA1015200) and Fundamental Research Funds for the Central Universities (SYSU, No. 24QNPY101).}

\begin{abstract}
Given a continuous function $\phi$ defined on a domain $\Omega\subset\RR^m\times\RR^n$, we show that if a Pr\'ekopa-type result holds for $\phi+\psi$ for any non-negative convex function $\psi$ on $\Omega$, then $\phi$ must be a convex function. Additionally, if the projection of $\Omega$ onto $\RR^m$ is convex, then $\overline{\Omega}$ is also convex. This provides a converse of Pr\'ekopa's theorem from convex analysis. We also establish analogous results for Berndtsson's theorem on the plurisubharmonic variation of Bergman kernels, showing that the plurisubharmonicity of weight functions and the pseudoconvexity of domains are necessary conditions in some sense.\\

\end{abstract}


\maketitle

{\small\tableofcontents}

\section{Introduction}

Motivated by recent developments in the converse $L^2$ theory, the present paper is devoted to proving converse results for Pr\'ekopa's theorem on the marginal integration of log-concave functions and Berndtsson's theorem on the plurisubharmonic variation of Bergman kernels.

Let $\phi(t,x)$ be a convex function on $\RR_t^m\times\RR_x^n$, then the \textit{minimum principle} for convex functions asserts that the marginal function
$$ \inf_x \phi(t,x) $$
is a convex function on $\RR_t^m$. In 1972, Pr\'ekopa established a stronger result:

\begin{theorem}[Pr\'{e}kopa \cite{Prekopa73}] \label{Thm:Prekopa}
Let $\phi$ be a convex function defined on a convex domain $\Omega\subset\RR_t^m\times\RR_x^n$. Let $p$ be the projection from $\RR^m\times\RR^n$ to $\RR^m$. For each $t\in p(\Omega)$, let $\Omega_t:=\{x\in\RR^n:(t,x)\in\Omega\}$. Then the function $\tilde{\phi}$ on $p(\Omega)$ defined by
$$ e^{-\tilde{\phi}(t)}:=\int_{\Omega_t} e^{-\phi(t,x)} d\lambda_x $$
is either convex or identically $-\infty$.
\end{theorem}

Pr\'ekopa's theorem implies the minimum principle by replacing $\phi$ with $k\phi+|x|^2$ and letting $k\to+\infty$. This remarkable result can be viewed as a functional version of the classic Brunn-Minkowski inequality (see \cite{Gardner} for more details).

In the theory of several complex variables, there is considerable interest in plurisubharmonic (\textit{psh} for short) functions and pseudoconvex domains, which are the complex analogs of convex functions and convex domains. Kiselman \cite{Kiselman} and Berndtsson \cite{Berndtsson98} respectively proved that the minimum principle and Pr\'ekopa-type results also hold for psh functions and pseudoconvex domains that satisfy certain symmetry properties.

In 2005, Berndtsson \cite{Berndtsson06} proved that the fiberwise Bergman kernel of a pseudoconvex domain is log-psh on the entire domain. This celebrated theorem generalizes the results of \cite{Berndtsson98}. Later, in the milestone paper \cite{Berndtsson09}, Berndtsson further established the Nakano positivity of certain direct image bundles, with the aforementioned theorem on fiberwise Bergman kernels corresponding to the Griffiths positivity of direct images. These results laid the foundation for the so-called ``\textit{complex Brunn-Minkowski theory}", which has since become a powerful tool in complex analysis, complex geometry, and algebraic geometry. For more details, we refer readers to Berndtsson's survey articles \cite{BerndtssonAbel, BerndtssonECM, BerndtssonICM} and the references therein.

Recall that, given an open set $D\subset\CC^n$ and an upper semi-continuous function $u$ on $D$, the weighted \textit{Bergman space} of $D$ is a Hilbert space defined by
$$ A^2(D;e^{-u}) := \left\{ f\in\calO(D): \|f\|^2 = \int_D |f|^2e^{-u} d\lambda < +\infty \right\}, $$
where $\calO(D)$ denotes the space of holomorphic functions on $D$ and $d\lambda$ denotes the Lebesgue measure. Moreover, the diagonal \textit{Bergman kernel} of $D$ is defined as
$$ B_D(z;e^{-u}) := \sup\left\{ |f(z)|^2: f\in A^2(D;e^{-u}), \|f\|\leq1 \right\}, \quad z\in D. $$
Since $\log|f|^2$ is psh for any $f\in\calO(D)$, one can easily show that $\log B_D(z;e^{-u})$ is a psh function on $D$. Now, Berndtsson's theorem on the psh variation of Bergman kernels can be stated as follows:

\begin{theorem}[Berndtsson \cite{Berndtsson06}] \label{Thm:Berndt}
Let $\Omega$ be a pseudoconvex domain in $\CC_\tau^m\times\CC_z^n$ and $\varphi$ a psh function on $\Omega$. For each $\tau$, let $\Omega_\tau:=\{z\in\CC^n:(\tau,z)\in\Omega\}$ and $\varphi_\tau:=\varphi(\tau,\cdot)$. Then
$$ (\tau,z) \mapsto \log B_{\Omega_\tau}(z;e^{-\varphi_\tau}) $$
is either psh or identically $-\infty$ on $\Omega$.
\end{theorem}

The case of $n=1$ and $\varphi\equiv0$ is due to Maitani-Yamaguchi \cite{MY04}. A key ingredient in Berndtsson's proof is H\"ormander's $L^2$-estimate for the $\bar\pd$-operator. Interestingly, Guan-Zhou \cite{GZ15} showed that their optimal $L^2$ extension theorem also implies Theorem \ref{Thm:Berndt}.

In $L^2$-existence theorems for $\bar\pd$-equations and (optimal) $L^2$ extension theorems, the weight functions are psh. Conversely, Deng-Ning-Wang \cite{DNW} showed that plurisubharmonicity is also a necessary condition in a certain sense (see \cite[etc]{DWZZ, HI21, DNWZ, DZ} for related work).

Motivated by this, we aim to investigate the converse of Pr\'ekopa's theorem and Berndtsson's theorem. Specifically, we inquire whether the convexity of $\phi$ and $\Omega$ in Theorem \ref{Thm:Prekopa} are necessary, and whether the plurisubharmonicity of $\varphi$ and the pseudoconvexity of $\Omega$ in Theorem \ref{Thm:Berndt} are necessary.\\

Imitating the definition of the ``optimal $L^2$ estimate property" in \cite{DNW}, we introduce an additional convex weight $\psi$ and require a Pr\'ekopa-type result for $\phi+\psi$. This leads to the following converse of Pr\'ekopa's theorem.

\begin{mainthm} \label{Thm:ConvPrekopa}
Let $\Omega$ be a domain in $\RR_t^m\times\RR_x^n$ and $\phi$ a continuous function on $\Omega$ such that $\int_{\Omega_t} e^{-\phi(t,\cdot)} d\lambda < +\infty$ for all $t\in p(\Omega)$, where $p$ is the projection from $\RR^m\times\RR^n$ to $\RR^m$ and $\Omega_t=\{x\in\RR^n:(t,x)\in\Omega\}$ are the fibers. Assume that for any convex function $\psi$ on $\Omega$ that is bounded below, the function $\tilde{\psi}$ defined by
\begin{equation} \label{Eq:ConvPrekopa}
    e^{-\tilde{\psi}(t)} := \int_{\Omega_t} e^{-\phi(t,\cdot)-\psi(t,\cdot)} d\lambda
\end{equation}
is also convex on $p(\Omega)$. Then $\phi$ must be a convex function. Moreover, if $p(\Omega)$ is convex, then $\overline{\Omega}$ is also convex.
\end{mainthm}

Here, we call a function $f:D\subset\RR^N\to\RR$ defined on a (not necessary convex) domain \textit{convex}, if the restriction of $f$ to any line segment $L\subset\Omega$ is convex. Since $\psi$ is bounded below, the integral in \eqref{Eq:ConvPrekopa} is convergent. As we will see in Section \ref{Sec:Prekopa-Ex}, Theorem \ref{Thm:ConvPrekopa} is false without the twisted factor $e^{-\psi}$, and we can only expect the convexity of $\overline{\Omega}$, rather than that of $\Omega$.

Let's explain the idea of the proof and the role of $e^{-\psi}$ in Theorem \ref{Thm:ConvPrekopa}. For simplicity, we consider the case of $m=1$. Given a line segment $L\subset\Omega$ defined by $x=At+b$, $t\in[t_0,t_1]$, we construct a sequence of lower-bounded convex functions $\psi_k$ on $\Omega$ such that $\tilde{\psi}_k$ defined by \eqref{Eq:ConvPrekopa} converges pointwise to $\phi(t,At+b)$. Then, the convexity of $\tilde{\psi}_k$ implies the convexity of $\phi|_L$. For a vertical line segment $L\subset\Omega$, the convexity of $\phi|_L$ follows from a simple limit argument.\\

Similarly, we consider an additional psh weight $\psi$ and require a Berndtsson-type result for $\varphi+\psi$. This leads to the following converse of Berndtsson's theorem.

\begin{mainthm} \label{Thm:ConvBerndt}
Let $\Omega$ be a domain in $\CC_\tau^{m}\times\CC_z^n$ and $\varphi$ a strongly upper semi-continuous function on $\Omega$ such that $B_{\Omega_\tau}(z;e^{-\varphi_\tau})$ is not identically zero on $\Omega$. Assume that for any psh function $\psi$ on $\Omega$ that is bounded below, the function
$$ (\tau,z) \mapsto \log B_{\Omega_\tau}(z;e^{-\varphi_\tau-\psi_\tau}) $$
is also psh on $\Omega$. Then $\varphi$ must be a psh function. Here, $\Omega_\tau:=\{z\in\CC^n:(\tau,z)\in\Omega\}$, $\varphi_\tau:=\varphi(\tau,\cdot)$ and $\psi_\tau:=\psi(\tau,\cdot)$.
\end{mainthm}

We call a function $u:D\subset\RR^N\to[-\infty,+\infty)$ \textit{strongly upper semi-continuous}, if
$$ \varlimsup_{x\in D\setminus S, x\to x_0} u(x) = u(x_0) $$
for any $x_0\in D$ and any subset $S\subset D$ of zero measure. It is known that psh functions are strongly upper semi-continuous.
Since $\psi$ is bounded below, we also have $B_{\Omega_\tau}(z;e^{-\varphi_\tau-\psi_\tau})\not\equiv0$ on $\Omega$. Similar to Theorem \ref{Thm:ConvPrekopa}, we have counterexamples showing that the additional psh weight $\psi$ in Theorem \ref{Thm:ConvBerndt} is necessary (see Section \ref{Sec:Berndt-Ex}). Moreover, the \textit{strong} upper semi-continuity of $\varphi$ is also necessary.

The twisted factor $e^{-\psi}$ plays the same role as in Theorem \ref{Thm:ConvPrekopa}. When $\varphi$ is continuous and $B_{\Omega_\tau}(z;e^{-\varphi_\tau})$ is nowhere zero on $\Omega$, the proof of Theorem \ref{Thm:ConvBerndt} is almost the same as that of Theorem \ref{Thm:ConvPrekopa}: given an affine map $z=A\tau+b$, where $A\in\mathrm{Mat}_\CC(n,m)$ and $b\in\CC^n$, we construct a sequence of lower-bounded psh functions $\psi_k$ on $\Omega$ such that $u_k(\tau):=\log B_{\Omega_\tau}(A\tau+b;e^{-\varphi_\tau-\psi_{k,\tau}})$ converges to $\varphi(\tau,A\tau+b)$; then the plurisubharmonicity of $u_k(\tau)$ implies the plurisubharmonicity of $\varphi(\tau,A\tau+b)$. However, in the general case, we encounter exceptional sets where this argument fails. This issue makes the proof more technical and tedious.

Recently, Li and Zhou \cite{LiZhou} proved a special case of Theorem \ref{Thm:ConvBerndt} using the characterization of psh functions via the multiple $L^2$-extension property. However, our approach is entirely different.\\

Now, we turn to the necessity of pseudoconvexity in Berndtsson's theorem. Note that if $\Omega$ is a domain in $\CC^m\times\CC^n$ such that $p^{-1}(L)\cap\Omega$ is pseudoconvex for every complex line $L\subset\CC^m$, where $p$ is the projection from $\CC^m\times\CC^n$ to $\CC^m$, then Theorem \ref{Thm:Berndt} still holds. Therefore, what really matters is the case where $m=1$. By extending the method used to prove Theorem \ref{Thm:ConvBerndt}, we derive the following converse of Berndtsson's theorem, which addresses the necessity of pseudoconvexity.

\begin{mainthm} \label{Thm:CharPsc}
Let $\Omega$ be a domain in $\CC_\tau\times\CC_z^n$ such that $\overline{\Omega_\tau}^\circ=\Omega_\tau$ for all $\tau\in p(\Omega)$, where $p$ is the projection from $\CC\times\CC^n$ to $\CC$ and $\Omega_\tau=\{z\in\CC^n:(\tau,z)\in\Omega\}$ are the fibers.
Assume that for any non-negative psh function $\psi(\tau,z)=\psi_\tau(z)$ on $\Omega$, the function
$$ (\tau,z) \mapsto \log B_{\Omega_\tau}(z;e^{-\psi_\tau}) $$
is either psh or identically $-\infty$ on $\Omega$. Then $\Omega$ must be a pseudoconvex domain.
\end{mainthm}

The regularity assumption $\overline{\Omega_\tau}^\circ=\Omega_\tau$ $(\forall\tau)$ is necessary. For example, Theorem \ref{Thm:Berndt} also holds for $\Omega=\triangle^2\setminus\{(0,0)\}$, but $\Omega$ is not pseudoconvex.

Recall that a domain $D\subset\CC^N$ is pseudoconvex if and only if $-\log d_D$ is psh on $D$, where $d_D$ denotes the Euclidean distance to $\pd D$. Therefore, the key to proving Theorem \ref{Thm:CharPsc} is constructing a sequence of non-negative psh functions $\psi_k$ on $\Omega$ such that the Bergman kernels $B_{\Omega_\tau}(z;e^{-\psi_{k,\tau}})$ provide sufficient information about $d_\Omega$.\\

\textbf{Acknowledgment.} The authors sincerely thank their PhD supervisor Prof. Xiangyu Zhou for his generous help over the years. The authors also want to thank Zhi Li, Zhuo Liu and Xujun Zhang for useful discussions on related topics.

\section{The necessity of convexity in Pr\'ekopa's theorem}  \label{Sec:Prekopa}

\subsection{The necessity of function convexity} \label{Sec:Prekopa-Fun}

We recall some necessary notations. Throughout this article,
$$ \BB^N(x_0;r) := \{x\in\RR^N: |x-x_0|<r \} $$
denotes an Euclidean ball in $\RR^N$. Therefore, a general ball in $\CC^n$ will be denoted by $\BB^{2n}(z_0;r)$. For simplicity, $\BB^N:=\BB^N(0;1)$ and $\sigma_N:=\textup{Vol}(\BB^N)$.

In the following, we prove the first part of Theorem \ref{Thm:ConvPrekopa}, i.e. $\phi$ is convex under the assumptions. Given two distinct points $(t_0,x_0),(t_1,x_1)\in\Omega$ such that the line segment connecting them is contained in $\Omega$, we need to prove
\begin{equation}\begin{aligned}
\phi\big(st_1+(1-s)t_0, sx_1+(1-s)&x_0\big) \\
\leq s\phi(t_1,x_1) + (1-s)&\phi(t_0,x_0), \quad \forall s\in(0,1).
\end{aligned}\end{equation}

Firstly, we consider the case of $t_0\neq t_1$. We may take a matrix $A\in\mathrm{Mat}_\RR(n,m)$ such that $A(t_1-t_0)=x_1-x_0$, then
$$ a:\RR^m\to\RR^n, \quad t\mapsto A(t-t_0)+x_0, $$
is an affine map satisfying $a(t_0)=x_0$ and $a(t_1)=x_1$. For convenience,
$$ U := \big\{t\in p(\Omega): (t,a(t))\in\Omega \big\}. $$
We want to construct a sequence of lower-bounded convex functions $\psi_k$ on $\Omega$ such that $\tilde{\psi}_k(t)$ converges pointwise to $\phi(t,a(t))$ on $U$, where $\tilde{\psi}_k$ is defined as in \eqref{Eq:ConvPrekopa}.

The construction is based on the following observation on ``localization": let $u$ be a continuous function on $\RR^n$, then
$$ \lim_{k\to+\infty} \int_{\RR^n} e^{-u(x)-k\max\{|x|-r,0\}} d\lambda = \int_{\BB^n(0;r)} e^{-u}d\lambda $$
for any $r>0$ and
$$ \lim_{r\to0^+} \frac{1}{\sigma_nr^n} \int_{\BB^n(0;r)} e^{-u} d\lambda = e^{-u(0)}. $$
We will apply a diagonal construction to achieve these two steps of localization simultaneously.

In the following, we define a sequence of lower-bounded convex functions on $\Omega$ by
\begin{equation}\label{Eq:psi_k}
\psi_k(t,x) := k^2\max\{|x-a(t)|-k^{-1},0\} + \log(\sigma_n k^{-n}), \quad k\in\ZZ_+.
\end{equation}
By assumption, the function $\tilde{\psi}_k$ defined by
\begin{equation}\label{Eq:tilde_psi_k}
e^{-\tilde{\psi}_k(t)} = \int_{\Omega_t} e^{-\phi(t,x)-\psi_k(t,x)} d\lambda_x 
\end{equation}
is convex on $p(\Omega)$. The key point of the proof is the following lemma.

\begin{lemma} \label{Lemma1}
For any $t\in U$, we have
$$ \lim_{k\to+\infty} \tilde{\psi}_k(t) = \phi(t,a(t)). $$
\end{lemma}

\begin{proof}
We choose $R>0$ such that $\overline{\BB^n(a(t);R)} \subset \Omega_t$. For $k>1/R$, we write
\begin{align*}
\int_{\Omega_t} e^{-\phi(t,x)-\psi_k(t,x)} d\lambda_x
& = \int_{\{|x-a(t)|\leq k^{-1}\}} + \int_{\{k^{-1}<|x-a(t)|<R\}} + \int_{\Omega_t\setminus\BB^n(a(t);R)} \\
& =: \mathbf{I} + \mathbf{II} + \mathbf{III}.
\end{align*}
Since $\psi_k(t,x)=\log(\sigma_nk^{-n})$ for $|x-a(t)|\leq k^{-1}$ and $\phi$ is continuous, we have
\begin{equation*}
\lim_{k\to+\infty} \mathbf{I} = \lim_{k\to+\infty} \frac{k^n}{\sigma_n} \int_{\{|x-a(t)|\leq k^{-1}\}} e^{-\phi(t,x)} d\lambda_x = e^{-\phi(t,a(t))}.
\end{equation*}
We need to show that the remaining terms $\mathbf{II}$ and $\mathbf{III}$ make no contribution. Let $C$ denotes the supremum of $e^{-\phi_t}$ over $\overline{\BB^n(a(t);R)}$, then
\begin{align*}
\mathbf{II} & \leq \frac{C}{\sigma_n} \int_{\{k^{-1}<|x-a(t)|<R\}} k^n e^{-k^2(|x-a(t)|-k^{-1})} d\lambda_x \\
& = nC \int_{1/k}^R k^n e^{-k^2r+k} r^{n-1} dr = nC \int_{1}^{kR} e^{-kr+k} r^{n-1} dr \\	
& \leq nC \int_0^\infty e^{-kr}(r+1)^{n-1} dr.
\end{align*}
By Lebesgue's dominated convergence theorem, $\lim_{k\to+\infty} \mathbf{II} = 0$. Notice that,
\begin{align*}
\mathbf{III} & = \frac{k^n}{\sigma_n} \int_{\Omega_t\setminus \BB^n(a(t);R)} e^{-\phi(t,x)-k^2(|x-a(t)|-k^{-1})} d\lambda_x \\
& \leq \frac{1}{\sigma_n} k^ne^{-k^2R+k} \int_{\Omega_t} e^{-\phi(t,x)} d\lambda_x.
\end{align*}
Since $\int_{\Omega_t} e^{-\phi_t} d\lambda < +\infty$, it follows that $\lim_{k\to+\infty} \mathbf{III} = 0$. In summary,
\begin{equation*}
\lim_{k\to+\infty} \int_{\Omega_t} e^{-\phi(t,x)-\psi_k(t,x)} d\lambda_x = e^{-\phi(t,a(t))}.
\end{equation*}
This completes the proof of the lemma.
\end{proof}

Since $\tilde{\psi}_k$ are convex functions on $p(\Omega)$ and the line segment connecting $t_0$ and $t_1$ is contained in $U\subset p(\Omega)$, we have
$$ \tilde{\psi}_k\big(st_1+(1-s)t_0\big) \leq s\tilde{\psi}_k(t_1)+(1-s)\tilde{\psi}_k(t_0), \quad \forall s\in(0,1). $$
Since $a(t)$ is affine, it follows from Lemma \ref{Lemma1} that
$$ \phi\big(st_1+(1-s)t_0, sx_1+(1-s)x_0\big) \leq s\phi(t_1,x_1) + (1-s)\phi(t_0,x_0), \quad \forall s\in(0,1). $$

Finally, we consider the case where $t_1=t_0$. Let $\eta=(1,0,\cdots,0)\in\RR^m$, then the triangle spanned by $(t_0,x_0)$, $(t_1,x_1)$ and $(t_1+\eps\eta,x_1)$ is contained in $\Omega$, whenever $\eps>0$ is small enough. Applying the above discussion to $(t_0,x_0)$ and $(t_1+\eps\eta,x_1)$, the desired inequality follows from a simple limit argument.

\subsection{The necessity of domain convexity} \label{Sec:Prekopa-Dom}

Since function convexity is a local property, Theorem \ref{Thm:Prekopa} remains true even if we replace $\Omega$ by $p^{-1}(U)\cap\Omega$, where $U$ is a domain in $\RR_t^m$ that is not necessarily convex. Therefore, when discussing the necessity of domain convexity, we always assume that $p(\Omega)$ is convex.

In the following, we prove the second part of Theorem \ref{Thm:ConvPrekopa}, namely that $\overline{\Omega}$ is convex, provided $p(\Omega)$ is convex. To prove the convexity of $\overline{\Omega}$, it suffices to show that if $(t_0,x_0),(t_1,x_1)\in\Omega$, then
\begin{equation} \label{Eq:MidPt}
(t^*,x^*) := \left(\frac{t_0+t_1}{2}, \frac{x_0+x_1}{2}\right) \in\overline{\Omega}. 
\end{equation}
The general case where $(t_0,x_0),(t_1,x_1)\in\overline{\Omega}$ will follow from a limit argument.

Firstly, we consider the case of $t_0\neq t_1$. We take an affine map $a:\RR^m\to\RR^n$ such that $a(t_0)=x_0$ and $a(t_1)=x_1$, then we define $\psi_k(t,x)$ and $\tilde{\psi}_k(t)$ by \eqref{Eq:psi_k} and \eqref{Eq:tilde_psi_k}. By assumption, $\tilde{\psi}_k$ are convex functions on $p(\Omega)$. Since $p(\Omega)$ is convex, the line segment connecting $t_0$ and $t_1$ is contained in $p(\Omega)$. Therefore,
$$ \tilde{\psi}_k(t^*) \leq \frac{\tilde{\psi}_k(t_0)+\tilde{\psi}_k(t_1)}{2}. $$
Letting $k\to+\infty$, it follows from Lemma \ref{Lemma1} that
$$ \varlimsup_{k\to+\infty} \tilde{\psi}_k(t^*) \leq \frac{\phi(t_0,x_0)+\phi(t_1,x_1)}{2} < +\infty. $$

Assume that \eqref{Eq:MidPt} is false, then we may choose $R>0$ such that $\overline{\BB^n(x^*;R)}$ is disjoint from $\Omega_{t^*}$. By the proof of Lemma \ref{Lemma1}, for $k>1/R$, we have
$$ e^{-\tilde{\psi}_k(t^*)} \leq \frac{1}{\sigma_n} k^ne^{-k^2R+k} \int_{\Omega_{t^*}}e^{-\phi_{t^*}} d\lambda. $$
Since $\int_{\Omega_{t^*}} e^{-\phi_{t^*}} d\lambda < +\infty$, it follows that
$$ \lim_{k\to+\infty} \tilde{\psi}_k(t^*) = +\infty. $$
This leads to a contradiction. Therefore, we prove \eqref{Eq:MidPt}.

If $t_0=t_1$, by applying the above discussion to $(t_0-\eta,x_0)$ and $(t_1+\eta,x_1)$, where $\eta\in\RR^m$ and $0<|\eta|\ll1$, we also have \eqref{Eq:MidPt}.

\subsection{Counterexamples} \label{Sec:Prekopa-Ex}

In this section, we construct some counterexamples concerning Theorem \ref{Thm:ConvPrekopa}. Firstly, we can only expect the convexity of $\overline{\Omega}$, rather than that of $\Omega$.

\begin{example} \label{Ex:ConvPrekopa1}
We consider the domain $\Omega:=\BB^{m+n}\setminus\{0\}$ in $\RR_t^m\times\RR_x^n$. Given any convex function $\phi(t,x)$ on $\Omega$, one can extend it into a convex function on $\BB^{m+n}$. Since a single point is negligible for Lebesgue integration, by Pr\'ekopa's theorem,
$$ -\log\left(\int_{\Omega_t}e^{-\phi_t}d\lambda\right) $$
is a convex function on $p(\Omega)=\BB^m$. However, $\Omega$ is not convex. 
\end{example}

Moreover, the additional convex weight $\psi$ in Theorem \ref{Thm:ConvPrekopa} is necessary.

\begin{example} \label{Ex:ConvPrekopa2}
Let $u(t)$ be a convex function on $U\subset\RR^m$ and $v(x)$ a continuous function on $D\subset\RR^n$. We consider the product domain $\Omega:=U\times D$ and the function $\phi(t,x):=u(t)+v(x)$. Clearly,
$$ -\log\left(\int_D e^{-\phi(t,\cdot)}d\lambda\right) = u(t) - \log\left(\int_D e^{-v}d\lambda\right) $$
is a convex function on $U$. However, $D$ and $\phi$ are not necessarily convex.
\end{example}

Notice that, if we reverse the roles of $U$ and $D$ in Example \ref{Ex:ConvPrekopa2}, and further require the convexity of
$$ -\log\left(\int_U e^{-\phi(\cdot,x)}d\lambda\right) = v(x) - \log\left(\int_U e^{-u}d\lambda\right) $$
then $\phi$ must be convex. This observation suggests that the additional convex weight $\psi$ in Theorem \ref{Thm:ConvPrekopa} might be unnecessary if we take into consideration fibrations in \textit{all} directions.

To be concrete, let $\phi$ be a continuous function on $\RR^2$. Given linearly independent vectors $\mathbf{e}_1,\mathbf{e}_2\in\RR^2$, then $L_t:=\{t\mathbf{e}_1+s\mathbf{e}_2:s\in\RR\}$ is a family of parallel lines parameterized by $t\in\RR$, which defines a fibration of $\RR^2$ along the $\mathbf{e}_2$-direction. If
$$ t \mapsto \int_{L_t}e^{-\phi} := \int_\RR e^{-\phi(t\mathbf{e}_1+s\mathbf{e}_2)} ds$$
is log-concave for any choice of $\mathbf{e}_1,\mathbf{e}_2\in\RR^2$, is $\phi$ necessarily convex on $\RR^2$?

This hypothesis appears to coincide with our knowledge on the \textit{inverse Radon transform}, which is closely related to medical imaging such as X-ray scans: it allows us to recover a suitable function defined on $\RR^2$ by knowing its integral over every line. However, we still have counterexamples.

To illustrate, we start with a radially symmetric convex function $\phi_0$ such that the marginal integration of $e^{-\phi_0}$ is \textit{strictly} log-concave. By perturbing $\phi_0$, we obtain another radially symmetric function $\phi$ that is not convex. If the perturbation is small enough, the marginal integration of $e^{-\phi}$ remains log-concave. Since $e^{-\phi}$ is radially symmetric, the marginal integration along any direction are equivalent.

Starting with $\phi_0(x,y)=x^2+y^2$ on $\RR^2$, we obtain the following counterexample.
 
\begin{example} \label{Ex:ConvPrekopa3}
Consider the radially symmetric function
\begin{equation}
    \phi(x,y)=|x^2+y^2-\eps^2|
\end{equation}
on $\RR^2$, where $\eps\in(0,1)$ is a constant. When $\eps$ is sufficiently small, for any linearly independent $\mathbf{e}_1,\mathbf{e}_2\in\RR^2$, the function
$$ t\mapsto -\log\left( \int_\RR e^{-\phi(t\mathbf{e}_1+s\mathbf{e}_2)} ds \right) $$
is convex on $\RR$. Even so, $\phi$ itself is not convex.
\end{example}

\begin{proof}
Since $\phi$ is radially symmetric on $\RR^2$, we may assume that $\mathbf{e}_1=(a,b)$ and $\mathbf{e}_2=(0,c)$, then
$$ \int_\RR e^{-\phi(t\mathbf{e}_1+s\mathbf{e}_2)} ds = \int_{\RR} e^{-\phi(at,bt+cs)} ds = \frac{1}{|c|} \int_\RR e^{-\phi(at,y)} dy. $$
Therefore, it suffices to show that, when $\eps\ll1$, the function $\tilde{\phi}$ defined by
$$ e^{-\tilde{\phi}(x)} = \int_\RR e^{-\phi(x,y)} dy $$
is convex on $\RR$. For $|x|\geq\eps$, we have
$$ \tilde{\phi}(x) = -\log\left(e^{\eps^2-x^2}\int_\RR e^{-y^2}dy\right) = x^2 - \log\left(e^{\eps^2}\sqrt{\pi}\right). $$
For $|x|<\eps$, we write
$$ F(x) := e^{-\tilde{\phi}(x)} = 2e^{\eps^2-x^2} A(x) + 2e^{x^2-\eps^2} B(x), $$
where
$$ A(x) := \int_{\sqrt{\eps^2-x^2}}^\infty e^{-y^2}dy \quad\text{and}\quad B(x) := \int_0^{\sqrt{\eps^2-x^2}} e^{y^2}dy. $$
Clearly, $\frac{\sqrt{\pi}}{2}-\eps \leq A(x) \leq \frac{\sqrt{\pi}}{2}$ and $0\leq B(x)\leq e\eps$. By direct computations,
\begin{align*}
F'(x) & = - 4xe^{\eps^2-x^2}A(x) + 4xe^{x^2-\eps^2}B(x), \\
F''(x) & = (8x^2-4)e^{\eps^2-x^2}A(x) + (8x^2+4)e^{x^2-\eps^2}B(x) - \frac{8x^2}{\sqrt{\eps^2-x^2}}.
\end{align*}
Consequently,
\begin{align*}
F'(x)F'(x)-F(x)F''(x) = &~ \frac{8x^2}{\sqrt{\eps^2-x^2}}F(x) + 8\big(e^{\eps^2-x^2}A(x)\big)^2 \\
&~ - 64x^2A(x)B(x) - 8\big(e^{x^2-\eps^2}B(x)\big)^2.
\end{align*}
When $\eps$ is sufficiently small, the right hand side is positive for $x\in(-\eps,\eps)$, and hence
\begin{equation*}
\tilde{\phi}''(x) = \frac{F'(x)F'(x)-F(x)F''(x)}{F(x)^2} > 0, \quad \forall x\in(-\eps,\eps).
\end{equation*}
Moreover, $\tilde{\phi}''(x)>0$ for $|x|>\eps$ and
\begin{equation*}
\lim_{x\to\eps^-}\tilde{\phi}'(x) = \lim_{x\to\eps^-} \frac{F'(x)F'(x)-F(x)F''(x)}{F(x)^2} = 2\eps =
\lim_{x\to\eps^+}\tilde{\phi}'(x).
\end{equation*}
We conclude that $\tilde{\phi}$ is a convex function on $\RR$.
\end{proof}

\subsection{A converse of the minimum principle}

Pr\'ekopa's theorem is a generalization of the minimum principle for convex functions. In this section, we prove a converse result for the minimum principle, which is parallel to Theorem \ref{Thm:ConvPrekopa}. 

\begin{theorem}
Let $\Omega$ be a domain in $\RR_t^m\times\RR_x^m$ and $\phi$ a continuous function on $\Omega$ such that $\inf_x \phi(t,x)$ is not identically $-\infty$ on $p(\Omega)$, where $p$ is the projection from $\RR^m\times\RR^n$ to $\RR^m$. Assume that for any non-negative convex function $\psi$ on $\Omega$, the function
$$ \inf_x \big( \phi(t,x)+\psi(t,x) \big) $$
is also convex on $p(\Omega)$. Then $\phi$ must be a convex function on $\Omega$. Moreover, if $p(\Omega)$ is convex, then $\overline{\Omega}$ is also convex.
\end{theorem}

\begin{proof}
Given two distinct points $(t_0,x_0),(t_1,x_1)\in\Omega$ such that the line segment connecting them is contained in $\Omega$, we need to prove
\begin{equation*}\begin{aligned}
\phi\big(st_1+(1-s)t_0, sx_1+(1-s)&x_0\big) \\
\leq s\phi(t_1,x_1) + (1-s)&\phi(t_0,x_0), \quad \forall s\in(0,1).
\end{aligned}\end{equation*}

Firstly, we consider the case of $t_0\neq t_1$. We may take an affine map $a:\RR^m\to\RR^n$ such that $a(t_0)=x_0$ and $a(t_1)=x_1$. By assumption,
$$ u_k(t) := \inf_x \big( \phi(t,x) + k|x-a(t)| \big), \quad k\in\ZZ_+, $$
are convex functions on $p(\Omega)$. We fix a point $t\in p(\Omega)$ such that $(t,a(t))\in\Omega$. Clearly, $u_k(t)\leq \phi(t,a(t))$. Given $\eps>0$, there exists a constant $\delta>0$ such that $\phi(t,x)\geq\phi(t,a(t))-\eps$ whenever $|x-a(t)|<\delta$. Since
\begin{gather*}\begin{cases}
\phi(t,x)+ k|x-a(t)| \geq \phi(t,a(t))-\eps, & |x-a(t)|<\delta, \\
\phi(t,x)+ k|x-a(t)| \geq \inf_x\phi(t,x) + k\delta, & |x-a(t)|\geq\delta,
\end{cases}\end{gather*}
we have $u_k(t) \geq \phi(t,a(t))-\eps$ for $k\gg1$. Here, we use the fact that $\inf_x \phi(t,x)>-\infty$ (because $\inf_x \phi(t,x)$ is convex on $p(\Omega)$ and is not identically $-\infty$). Since $\eps>0$ is arbitrary, we have
$$ \lim_{k\to+\infty} u_k(t) = \phi(t,a(t)). $$
Then the convexity of $u_k$ implies the desired inequality. The case of $t_0=t_1$ follows from a simple limit argument.

In the following, we assume that $p(\Omega)$ is convex. If $\overline{\Omega}$ is not convex, then there exist $(t_0,x_0),(t_1,x_1)\in\Omega$ such that their mid-point $(t^*,x^*):=(\frac{t_0+t_1}{2},\frac{x_0+x_1}{2})$ is not contained in $\overline{\Omega}$. After a slight perturbation, we may assume that $t_0\neq t_1$. Let $a(t)$ and $u_k(t)$ be as before, then $a(t^*)=x^*$,
$$ \lim_{k\to+\infty}u_k(t_0)=\phi(t_0,x_0) \quad\text{and}\quad \lim_{k\to+\infty}u_k(t_1)=\phi(t_1,x_1). $$
Since $u_k$ are convex,
$$ \varlimsup_{k\to\infty} u_k(t^*) \leq \frac{\phi(t_0,x_0)+\phi(t_1,x_1)}{2} < +\infty. $$
Since $(t^*, x^*)\notin\overline{\Omega}$, there exists a constant $R>0$ such that $|x-x^*|\geq R$ for any $(t^*,x)\in\Omega$, and then
$$ u_k(t^*)\geq \inf_x \phi(t^*, x)+kR. $$
Consequently, $\lim_{k\to+\infty} u_k(t^*)=+\infty$, which leads to a contradiction.
\end{proof}

Similar to Theorem \ref{Thm:ConvPrekopa}, the above theorem is false without the additional convex function $\psi$. We consider the function
$$ \phi(t,x)=|t^2+x^2-1| $$
on $\RR^2$, then $\inf_x\phi(t,x) = \max\{t^2-1,0\}$ is convex on $\RR$. However, $\phi$ itself is not convex.

\section{The necessity of plurisubharmonicity in Berndtsson's theorem} \label{Sec:Berndt}

\subsection{A special case} \label{Sec:SpecialCase}

To highlight the key idea of the proof, in this section, we will prove Theorem \ref{Thm:ConvBerndt} under two additional assumptions:
\begin{itemize}
    \item[(i)] $\varphi$ is continuous on $\Omega$;
    \item[(ii)] the fiberwise Bergman kernel $B_{\Omega_\tau}(z;e^{-\varphi_\tau})$ is nowhere zero on $\Omega$.
\end{itemize}
We will use these two conditions as less as possible. Notice that, if $\Omega$ is bounded and $\varphi$ is continuous up to the boundary, then the second condition is satisfied.

The proof is parallel to that of Theorem \ref{Thm:ConvPrekopa}. To prove the plurisubharmonicity of $\varphi$, we need to verify the following mean-value inequality
\begin{equation} \label{Eq:SubMeanC0}
	\varphi(\tau_0,z_0) \leq \frac{1}{2\pi} \int_0^{2\pi} \varphi(\tau_0+e^{i\theta}\eta, z_0+e^{i\theta}\xi) d\theta
\end{equation}
for any \textit{affine analytic disc}
$$ (\tau_0,z_0)+(\eta,\xi)\overline{\triangle} := \big\{ (\tau_0+w\eta, z_0+w\xi): w\in\overline{\triangle} \big\} $$
contained in $\Omega$, where $(\tau_0,z_0)\in\Omega$ is the center, $(\eta,\xi)\in\CC^m\times\CC^n$ is the radius and $\overline{\triangle}:=\{w\in\CC:|w|\leq1\}$ is the closed unit disc in $\CC$.

Firstly, we assume that $\eta\neq0$, then we can choose a matrix $A\in\mathrm{Mat}_\CC(n,m)$ such that $A\eta=\xi$. Define
$$ a(\tau):=z_0+A(\tau-\tau_0) \quad\text{and}\quad U := \{\tau\in\CC^m: (\tau,a(\tau)) \in \Omega \}. $$
Notice that the closed disc $\tau_0+\eta\overline{\triangle}$ is contained in $U$. As mentioned in the Introduction, we want to construct a sequence of lower-bounded psh functions $\psi_k$ on $\Omega$ such that $u_k(\tau):=\log B_{\Omega_\tau}(a(\tau);e^{-\varphi_\tau-\psi_{k,\tau}})$ converges to $\varphi(\tau,a(\tau))$. The construction is based on the following results on ``localization":
\begin{itemize}
    \item (Lemma 3.1 of \cite{Berndtsson06}) Let $D_0$ and $D_1$ be bounded domains in $\CC^n$, with $D_0\Subset D_1$. Let $\{\phi_j\}_{j=0}^\infty$ be a sequence of continuous functions on $D_1$ such that $\phi_j=\phi$ in $\overline{D_0}$ and that $\phi_j\nearrow+\infty$ almost everywhere in $D_1\setminus\overline{D_0}$. Assume that $A^2(D_1;e^{-\phi_0})$ is dense in $A^2(D_0;e^{-\phi_0})$, then $B_{D_1}(z;e^{-\phi_j})$ increases to $B_{D_0}(z;e^{-\phi})$ for any $z\in D_0$.
    \item Let $\phi$ be a continuous function defined on a domain $D\subset\CC^n$, then
    $$ \lim_{r\to0^+} \sigma_{2n}r^{2n} B_{\BB^{2n}(z;r)}(z;e^{-\phi}) = e^{\phi(z)}, \quad \forall z\in D. $$
\end{itemize}
As in Section \ref{Sec:Prekopa-Fun}, we will apply a diagonal construction to achieve these two steps of localization simultaneously.

In the following, we define a sequence of psh functions on $\Omega$,
\begin{equation} \label{Eq:PshTwist}
	\psi_k(\tau,z) := k\max\{ \log(k|z-a(\tau)|),0 \} + \log(\sigma_{2n}k^{-2n}), \quad k\in\ZZ_+.
\end{equation}
Since $\psi_k$ is bounded below, by assumption,
$$ (\tau,z) \mapsto \log B_{\Omega_\tau}(z;e^{-\varphi_\tau-\psi_{k,\tau}}) $$
is a psh function on $\Omega$. We have the following counterpart of Lemma \ref{Lemma1}.

\begin{lemma} \label{Lemma2}
For any $\tau\in U$, one has
\begin{equation} \label{Eq:Lemma2}
\lim_{k\to+\infty} B_{\Omega_\tau}(a(\tau);e^{-\varphi_\tau-\psi_{k,\tau}}) = e^{\varphi(\tau,a(\tau))}.
\end{equation}
\end{lemma}

\begin{proof}
Since $(\tau,a(\tau))\in\Omega$, we may choose $r>0$ such that $\overline{\BB^{2n}(a(\tau);r)} \subset \Omega_\tau$.

Since $\varphi$ is upper semi-continuous, for any constant $c>\varphi(\tau,a(\tau))$, we can find $r_c\in(0,r)$ such that $\varphi(\tau,z)<c$ for all $z\in\BB^{2n}(a(\tau);r_c)$. Given $k>1/r_c$ and $f\in A^2(\Omega_\tau;e^{-\varphi_\tau-\psi_{k,\tau}})$, since $\psi_k(\tau,z) = \log(\sigma_{2n}k^{-2n})$ for $|z-a(\tau)|<1/k$, we have
\begin{align*}
\int_{\Omega_\tau} |f(z)|^2 e^{-\varphi(\tau,z) - \psi_k(\tau,z)} d\lambda_z
& \geq \frac{k^{2n}}{\sigma_{2n}} \int_{\BB^{2n}(a(\tau);1/k)} |f(z)|^2e^{-\varphi(\tau,z)} d\lambda_z \\
& \geq \frac{k^{2n}}{\sigma_{2n}} \int_{\BB^{2n}(a(\tau);1/k)} |f(z)|^2 e^{-c} d\lambda_z \\
& \geq  |f(a(\tau))|^2 e^{-c}.
\end{align*}
In the last inequality, we used the plurisubharmonicity of $|f|^2$. By the definition of Bergman kernels,
$$ B_{\Omega_\tau}(a(\tau);e^{-\varphi_\tau-\psi_{k,\tau}}) \leq e^c, \quad \forall k>1/r_c. $$
Since $c>\varphi(\tau,a(\tau))$ is arbitrary, we conclude that
\begin{equation} \label{Eq:UB}
\varlimsup_{k\to+\infty} B_{\Omega_\tau}(a(\tau);e^{-\varphi_\tau-\psi_{k,\tau}}) \leq e^{\varphi(\tau,a(\tau))}.
\end{equation}
	
\textbf{\itshape By condition (ii)}, $B_{\Omega_\tau}(a(\tau);e^{-\varphi_\tau})>0$, then there exists a holomorphic function $g\in A^2(\Omega_\tau;e^{-\varphi_\tau})$ with $g(a(\tau))=1$. In particular, $e^{-\varphi_\tau}$ is locally integrable near $a(\tau)\in\Omega_\tau$. For convenience, we set
\begin{equation}
    \delta_k := k^{\frac{2n+1}{k}-1} = (\sqrt[k]{k})^{2n+1}/k.
\end{equation}
Clearly, $\delta_k>1/k$ and $\lim_{k\to+\infty}k\delta_k=1$. For $k\gg1$ such that $\delta_k<r$, we write
\begin{align*}
\int_{\Omega_\tau} |g|^2 e^{-\varphi_\tau - \psi_{k,\tau}} d\lambda = \int_{\{|z-a(\tau)|<\delta_k\}}
+ \int_{\Omega_\tau\cap \{|z-a(\tau)|\geq\delta_k\}} =: \mathbf{I} + \mathbf{II}.
\end{align*}
Notice that, if $|z-a(\tau)|\geq\delta_k$, then
\begin{align*}
\psi_k(\tau,z) & = k\log(k|z-a(\tau)|) + \log(\sigma_{2n}k^{-2n}) \\
& \geq k\log(k^{\frac{2n+1}{k}}) + \log(\sigma_{2n}k^{-2n}) = \log(\sigma_{2n}k).
\end{align*}
Consequently,
$$ \mathbf{II} \leq \frac{1}{\sigma_{2n}k} \int_{\Omega_\tau\cap\{|z-a(\tau)|\geq\delta_k\}} |g(z)|^2 e^{-\varphi(\tau,z)} d\lambda_z. $$
Since $\int_{\Omega_\tau} |g|^2e^{-\varphi_\tau} d\lambda < +\infty$, it follows that $\lim_{k\to+\infty} \mathbf{II} = 0$. On the other hand, since $\psi_k\geq\log(\sigma_{2n}k^{-2n})$, we have
\begin{align*}
\mathbf{I} & \leq \frac{k^{2n}}{\sigma_{2n}} \int_{\{|z-a(\tau)|<\delta_k\}} |g(z)|^2e^{-\varphi(\tau,z)} d\lambda_z \\
& = (k\delta_k)^{2n} \frac{1}{\sigma_{2n}\delta_k^{2n}} \int_{\{|z-a(\tau)|<\delta_k\}} |g(z)|^2e^{-\varphi(\tau,z)} d\lambda_z.
\end{align*}
\textbf{\itshape By condition (i)}, $|g|^2e^{-\varphi_\tau}$ is continuous. Since $\lim_{k\to+\infty}\delta_k=0$ and $g(a(\tau))=1$, it follows that
\begin{equation*}
\lim_{k\to+\infty} \frac{1}{\sigma_{2n}\delta_k^{2n}} \int_{\{|z-a(\tau)|<\delta_k\}} |g(z)|^2e^{-\varphi(\tau,z)} d\lambda_z = e^{-\varphi(\tau,a(\tau))}.
\end{equation*}
As $\lim_{k\to+\infty} k\delta_k = 1$, we see that $\varlimsup_{k\to+\infty} \mathbf{I} \leq  e^{-\varphi(\tau,a(\tau))}$. In summary,
$$ \varlimsup_{k\to+\infty} \int_{\Omega_\tau} |g|^2 e^{-\varphi_\tau - \psi_{k,\tau}} d\lambda \leq e^{-\varphi(\tau,a(\tau))}. $$

By the definition of Bergman kernels,
$$ B_{\Omega_\tau}(a(\tau);e^{-\varphi_\tau-\psi_{k,\tau}}) \geq \frac{1}{\int_{\Omega_\tau}|g|^2e^{-\varphi_\tau-\psi_{k,\tau}}d\lambda}. $$
Consequently,
\begin{equation} \label{Eq:LB}
\varliminf_{k\to+\infty} B_{\Omega_\tau}(a(\tau);e^{-\varphi_\tau-\psi_{k,\tau}}) \geq e^{\varphi(\tau,a(\tau))}.
\end{equation}

Combining \eqref{Eq:UB} and \eqref{Eq:LB}, we complete the proof of the lemma.
\end{proof}

Now we are ready to prove the mean-value inequality \eqref{Eq:SubMeanC0}. Recall that,
\begin{equation} \label{Eq:ukDef}
	u_k(\tau) := \log B_{\Omega_\tau}(a(\tau);e^{-\varphi_\tau-\psi_{k,\tau}})
\end{equation}
are psh functions on $U$. For any compact subset $K$ of $U$, there exists some $\eps>0$ such that
$$ \tilde{K} := \big\{ (\tau,z)\in\CC^m\times\CC^n: \tau\in K, |z-a(\tau)|\leq\eps \big\} $$
is contained in $\Omega$. Since $\varphi$ is upper semi-continuous, the supremum $c:=\sup_{\tilde{K}}\varphi$ is finite. Similar to the proof of \eqref{Eq:UB}, we can show that $u_k\leq c$ on $K$ for all $k>1/\eps$. Therefore, $u_k$ are locally uniformly bounded from above on $U$.

According to Lemma \ref{Lemma2}, $\lim_{k\to+\infty} u_k(\tau) = \varphi(\tau,a(\tau))$ for every $\tau\in U$. By the construction, $a(\tau_0+w\eta)=z_0+w\xi$ for any $w\in\CC$. Since $u_k$ are psh and locally uniformly bounded from above on $U$, it follows from Fatou's lemma that
\begin{align*}
	\varphi(\tau_0,z_0) = \lim_{k\to+\infty} u_k(\tau_0) & \leq \varlimsup_{k\to+\infty} \frac{1}{2\pi} \int_0^{2\pi} u_k(\tau_0+e^{i\theta}\eta) d\theta \\
	& \leq \frac{1}{2\pi} \int_0^{2\pi} \lim_{k\to+\infty} u_k(\tau_0+e^{i\theta}\eta) d\theta \\
	& = \frac{1}{2\pi} \int_0^{2\pi} \varphi(\tau_0+e^{i\theta}\eta, z_0+e^{i\theta}\xi) d\theta.
\end{align*}
Therefore, we prove \eqref{Eq:SubMeanC0} in the case of $\eta\neq0$.

We remain to prove \eqref{Eq:SubMeanC0} in the case of $\eta=0$. We take a sequence $\{\eta_j\}_{j=1}^\infty$ in $\CC^m\setminus\{0\}$ such that $\eta_j\to0$ and $(\tau_0,z_0)+(\eta_j,\xi)\overline{\triangle} \subset\Omega$ for all $j$. By Fatou's lemma and the upper semi-continuity of $\varphi$,
\begin{align*}
	\varphi(\tau_0,z_0) & \leq \varlimsup_{j\to+\infty} \frac{1}{2\pi} \int_0^{2\pi} \varphi(\tau_0+e^{i\theta}\eta_j, z_0+e^{i\theta}\xi) d\theta \\
	& \leq \frac{1}{2\pi} \int_0^{2\pi} \varlimsup_{j\to+\infty} \varphi(\tau_0+e^{i\theta}\eta_j, z_0+e^{i\theta}\xi) d\theta \\
	& \leq \frac{1}{2\pi} \int_0^{2\pi} \varphi(\tau_0, z_0+e^{i\theta} \xi) d\theta.
\end{align*}
This completes the proof.

\subsection{The general case} \label{Sec:GeneralCase}

In the previous section, we proved Theorem \ref{Thm:ConvBerndt} under two additional assumptions. Here, we consider the general case.

Recall that conditions (i) and (ii) were used in the proof of Lemma \ref{Lemma2}. Without these two assumptions, the conclusion of Lemma \ref{Lemma2} may fail at certain points. Nevertheless, we will show that the exceptional set has zero measure.

\begin{lemma} \label{Lemma:ZeroMes1}
Given a sequence $\{\delta_k\}_{k=1}^\infty$ of positive numbers decreasing to $0$, there exists a set $S\subset\Omega$ of zero measure such that, for every $(\tau,z)\in\Omega\setminus S$, one has $B_{\Omega_\tau}(z;e^{-\varphi_\tau})>0$ and
\begin{equation} \label{Eq:LebPt}
\lim_{k\to+\infty} \frac{1}{\sigma_{2n}\delta_k^{2n}} \int_{\BB^{2n}(z;\delta_k)} e^{-\varphi(\tau,w)} d\lambda_w = e^{-\varphi(\tau,z)}.
\end{equation}
\end{lemma}

Lemma \ref{Lemma:ZeroMes1} is an improved version of the Lebesgue differentiation theorem. The tedious proof will be left to the next subsection.

In the following, we shall take $\delta_k := k^{\frac{2n+1}{k}-1}$ as in the proof of Lemma \ref{Lemma2}. We aim to prove the plurisubharmonicity of $\varphi$ at a given point $(\tau_0,z_0)\in\Omega\setminus S$.

We fix a vector $(\eta,\xi)\in\CC^m\times\CC^n$ with $\eta\neq0$ and $(\tau_0,z_0)+(\eta,\xi)\overline{\triangle} \subset \Omega$. We choose a matrix $A\in\mathrm{Mat}_\CC(n,m)$ such that $A\eta=\xi$. Let $a(\tau):=z_0+A(\tau-\tau_0)$ and $U := \{\tau\in\CC^m: (\tau,a(\tau)) \in \Omega \}$. We define lower-bounded psh functions
$$ \psi_k(\tau,z) := k\max\{\log(k|z-a(\tau)|),0\} + \log(\sigma_{2n}k^{-2n}) $$
as in \eqref{Eq:PshTwist}, then
$$ u_k(\tau) := \log B_{\Omega_\tau}(a(\tau);e^{-\varphi_\tau-\psi_{k,\tau}}) $$
are psh functions on $U$. By examining the proof of Lemma \ref{Lemma2}, we see that
\begin{itemize}
\item $u_k$ are locally uniformly bounded from above on $U$;
\item for any $\tau\in U$, one has $\varlimsup_{k\to+\infty}u_k(\tau)\leq\varphi(\tau,a(\tau))$;
\item for any $\tau\in U$ such that $(\tau,a(\tau))\notin S$, one has $\lim_{k\to+\infty}u_k(\tau)=\varphi(\tau,a(\tau))$.
\end{itemize}

Since $(\tau_0,z_0)\notin S$, $\tau_0+\eta\overline{\triangle} \subset U$, $u_k$ are psh and locally uniformly bounded from above on $U$, it follows from Fatou's lemma that
\begin{align*}
	\varphi(\tau_0,z_0) = \lim_{k\to+\infty} u_k(\tau_0) & \leq \varlimsup_{k\to+\infty} \frac{1}{2\pi} \int_0^{2\pi} u_k(\tau_0+e^{i\theta}\eta) d\theta \\
	& \leq \frac{1}{2\pi} \int_0^{2\pi} \varlimsup_{k\to+\infty} u_k(\tau_0+e^{i\theta}\eta) d\theta \\
        & \leq \frac{1}{2\pi} \int_0^{2\pi} \varphi(\tau_0+e^{i\theta}\eta,z_0+e^{i\theta}\xi) d\theta.
\end{align*}
Therefore, the mean-value inequality \eqref{Eq:SubMeanC0} holds for any affine analytic disc $(\tau_0,z_0) + (\eta,\xi)\overline{\triangle} \subset\Omega$ with $\eta\neq0$. Since $\varphi$ is upper semi-continuous, the same inequality holds for any affine analytic disc in $\Omega$ centered at $(\tau_0,z_0)$ (see the last paragraph of Section \ref{Sec:SpecialCase}). In conclusion, $\varphi$ is psh at $(\tau_0,z_0)\in\Omega\setminus S$.

We remain to prove the plurisubharmonicity at $(\tau_*,z_*)\in S$. Since $S\subset\Omega$ has zero measure, by the \textbf{\itshape strong} upper semi-continuity of $\varphi$, we can choose a sequence of points $(\tau_j,z_j)$ in $\Omega\setminus S$ such that
$$ \lim_{j\to+\infty} (\tau_j,z_j) = (\tau_*,z_*) \quad\text{and}\quad \lim_{j\to+\infty} \varphi(\tau_j,z_j) = \varphi(\tau_*,z_*). $$
Given any affine analytic disc $(\tau_*,z_*)+(\eta,\xi)\overline{\triangle} \subset \Omega$ centered at $(\tau_*,z_*)$, it follows from Fatou's lemma and the upper semi-continuity of $\varphi$ that
\begin{align*}
	\varphi(\tau_*,z_*) = \lim_{j\to+\infty} \varphi(\tau_j,z_j) & \leq
	\frac{1}{2\pi } \varlimsup_{j\to+\infty} \int_0^{2\pi} \varphi(\tau_j+e^{i\theta}\eta,z_j+e^{i\theta}\xi) d\theta \\
	& \leq \frac{1}{2\pi} \int_0^{2\pi} \varlimsup_{j\to+\infty} \varphi(\tau_j+e^{i\theta}\eta,z_j+e^{i\theta}\xi) d\theta \\
	& \leq \frac{1}{2\pi} \int_0^{2\pi} \varphi(\tau_*+e^{i\theta}\eta,z_*+e^{i\theta}\xi) d\theta.
\end{align*}
Therefore, $\varphi$ is also psh at $(\tau_*,z_*)\in S$. This completes the proof of Theorem \ref{Thm:ConvBerndt}.

\begin{remark} \label{Rmk:aeConvBerndt}
In the above proof of Theorem \ref{Thm:ConvBerndt}, the \textit{strong} upper semi-continuity of $\varphi$ is only used in the last paragraph. Therefore, if we assume only that $\varphi$ is upper semi-continuous, then the same argument shows the following: {\itshape there exists a set $S\subset\Omega$ of zero measure such that, for any affine analytic disc $(\tau_0,z_0) +(\eta,\xi)\overline{\triangle} \subset \Omega$ centered at $(\tau_0,z_0)\in\Omega\setminus S$, the mean-value inequality \eqref{Eq:SubMeanC0} holds.} By Lemma \ref{Lemma:aePSH} below, we can conclude that there exists a psh function $\tilde{\varphi}$ on $\Omega$ such that $\varphi=\tilde{\varphi}$ on $\Omega\setminus S$.
\end{remark}

\begin{lemma} \label{Lemma:aePSH}
Let $u$ be an upper semi-continuous function on a domain $D\subset\CC^N$ such that $u>-\infty$ almost everywhere. Assume there exists a set $S\subset D$ with zero measure such that
\begin{equation}\label{Eq:aePSH}
    u(z) \leq \frac{1}{2\pi} \int_0^{2\pi} u(z+e^{i\theta}\eta) d\theta
\end{equation}
for any affine analytic disc $z+\eta\overline{\triangle} \subset D$ centered at $z\in D\setminus S$. Then there exists a psh function $\tilde{u}$ on $D$ such that $u=\tilde{u}$ on $D\setminus S$.
\end{lemma}

\begin{proof}
Given $x\in D$, let $r:=d_D(x)/2$, then there exists a point $z\in D\setminus S$ with $u(z)>-\infty$ and $|z-x|<r$. Integrating \eqref{Eq:aePSH} with respect to $\eta\in\BB^{2N}(0;r)$, we see that
$$ -\infty < u(z) \leq \frac{1}{\sigma_{2N}r^{2N}} \int_{\BB^{2N}(z;r)} u d\lambda.  $$
As $u$ is locally bounded from above, this implies $u$ is integrable near $x$. Therefore, $u\in L_{\textup{loc}}^1(D)$.
We choose a non-negative function $\chi\in C_c^\infty(\CC^N)$ such that $\chi(z)$ depends only on $|z|$, $\chi(z)=0$ for $|z|\geq1$, and $\int_{\CC^N} \chi d\lambda=1$. For any $\eps>0$, we set $\chi_\eps(z):=\eps^{-2N}\chi(z/\eps)$. Then
$$ u_\eps(z) = (u*\chi_\eps)(z) := \int_{\CC^N} u(z-\eps\zeta)\chi(\zeta) d\lambda_\zeta $$
is a smooth function on $D_\eps := \{z\in D: d_D(z)>\eps\}$. For any $z\in D_\eps$ and $\eta\in\CC^N$ such that $z+\eta\overline{\triangle}\subset D_\eps$, since $z-\eps\zeta\notin S$ for almost every $\zeta\in\CC^N$, it is clear that
\begin{align*}
u_\eps(z) & \leq \int_{\CC^N} \left( \frac{1}{2\pi} \int_0^{2\pi} u(z-\eps\zeta+e^{i\theta}\eta) d\theta  \right) \chi(\zeta) d\lambda_\zeta \\
& = \frac{1}{2\pi} \int_0^{2\pi} \left( \int_{\CC^N} u(z-\eps\zeta+e^{i\theta}\eta) \chi(\zeta) d\lambda_\zeta \right) d\theta \\ & = \frac{1}{2\pi} \int_0^{2\pi} u_\eps(z+e^{i\theta}\eta) d\theta.
\end{align*}
Therefore, $u_\eps$ is a psh function on $D_\eps$. By a similar proof to \cite[Theorem 1.6.11]{HorBook}, $u_\eps*\chi_\delta$ is decreasing as $\delta\searrow0$. Since $u_\eps*\chi_\delta = u_\delta*\chi_\eps$ and $\lim_{\eps\to0} u_\delta*\chi_\eps = u_\delta$, letting $\eps\searrow0$, we conclude that $u_\delta$ is decreasing as $\delta\searrow0$. Since $u$ is upper semi-continuous, $\varlimsup_{\eps\to0} u_\eps(z) \leq u(z)$ for any $z\in D$. Moreover, by integrating \eqref{Eq:aePSH} with respect to $\eta$, we see that $u_\eps(z)\geq u(z)$ for any $z\in D\setminus S$. In summary, $\tilde{u} := \lim_{\eps\to0} u_\eps$ is a psh function on $D$ and $u=\tilde{u}$ on $D\setminus S$.
\end{proof}

\subsection{The exceptional set} \label{Sec:ExceptionalSets}

In this section, we prove Lemma \ref{Lemma:ZeroMes1} that appears in the proof of Theorem \ref{Thm:ConvBerndt}. The technical proof relies on the Lebesgue differentiation theorem and the Fubini theorem. Note that, before applying the Fubini theorem, we must first verify that the relevant sets are measurable.

Recall that, given a function $f\in L_{\textup{loc}}^1(\RR^N)$, any $x_0\in\RR^N$ satisfying
\begin{equation*}
    \lim_{\eps\to0} \frac{1}{\sigma_N\eps^N} \int_{\BB^N(x_0;\eps)} |f-f(x_0)| d\lambda = 0
\end{equation*}
is called a \textit{Lebesgue point} of $f$. The Lebesgue differentiation theorem says that almost every $x\in\RR^N$ is a Lebesgue point of $f\in L_{\textup{loc}}^1(\RR^N)$.

\begin{proof}[The proof of Lemma \ref{Lemma:ZeroMes1}]
Taking $\psi\equiv0$, the assumption of Theorem \ref{Thm:ConvBerndt} says $\Phi := \log B_{\Omega_\tau}(z;e^{-\varphi_\tau})$ is a psh function on $\Omega$. Since $\Phi\not\equiv-\infty$, the pluripolar set
$$ S_0 := \Phi^{-1}(-\infty) = \left\{ (\tau,z)\in\Omega: B_{\Omega_\tau}(z;e^{-\varphi_\tau})=0 \right\} $$
has zero measure in $\Omega$. We denote by $p$ the projection from $\CC^m\times\CC^n$ to $\CC^m$. By Fubini's theorem, the set
$$ P := \left\{ \tau\in p(\Omega): \int_{\Omega_\tau} \mathbb{I}_{S_0}(\tau,z) d\lambda_z>0 \right\} $$
has zero measure in $p(\Omega)$.

For any positive integer $k$, we define a measurable function $\varphi_k$ on $\Omega$ by
$$ e^{-\varphi_k(\tau,z)} := \frac{1}{\sigma_{2n}\delta_k^{2n}} \int_{\BB^{2n}(z;\delta_k)} \mathbb{I}_\Omega(\tau,w) e^{-\varphi(\tau,w)} d\lambda_w. $$
Clearly, the limits
\begin{equation*}
\varphi^*(\tau,z) := \varlimsup_{k\to+\infty} \varphi_k(\tau,z) \quad\text{and}\quad
\varphi_*(\tau,z) := \varliminf_{k\to+\infty} \varphi_k(\tau,z)
\end{equation*}
are measurable functions on $\Omega$. Then
$$ S_1 := \big\{ (\tau,z)\in\Omega: \varphi^*(\tau,z) \neq \varphi(\tau,z) \text{ or } \varphi_*(\tau,z) \neq \varphi(\tau,z) \big\} $$
is a \textit{measurable} set. We claim that $S_1$ has zero measure in $\Omega$, then Lemma \ref{Lemma:ZeroMes1} follows by setting $S:=S_0\cup S_1$.

Given $\tau\in p(\Omega)\setminus P$, since $B_{\Omega_\tau}(\cdot;e^{-\varphi_\tau})\not\equiv0$, there exists some nontrivial holomorphic function $g\in\calO(\Omega_\tau)$ such that
$$ \int_{\Omega_\tau} |g|^2e^{-\varphi_\tau} d\lambda < +\infty. $$
In particular, $e^{-\varphi_\tau}$ is locally integrable on $\Omega_\tau\setminus\{g=0\}$. By the Lebesgue differentiation theorem, for almost every $z\in \Omega_\tau\setminus\{g=0\}$, we have
$$ \lim_{\delta\to0} \frac{1}{\sigma_{2n}\delta^{2n}} \int_{\BB^{2n}(z;\delta)} \mathbb{I}_\Omega(\tau,w) e^{-\varphi(\tau,w)} d\lambda_w = e^{-\varphi(\tau,z)}. $$
Consequently,
$$ \varphi^*(\tau,z) = \varphi_*(\tau,z) = \lim_{k\to+\infty} \varphi_k(\tau,z) = \varphi(\tau,z). $$
for almost every $z\in \Omega_\tau\setminus\{g=0\}$. Therefore, whenever $\tau\in p(\Omega)\setminus P$, the slice $\{z\in\CC^n: (\tau,z)\in S_1\}$ has zero measure. By Fubini's theorem again, $S_1\subset\Omega$ has zero measure. This completes the proof of Lemma \ref{Lemma:ZeroMes1}.
\end{proof}

\subsection{Counterexamples} \label{Sec:Berndt-Ex}

In this section, we construct some counterexamples concerning Theorem \ref{Thm:ConvBerndt}. Firstly, the \textit{strong} upper semi-continuity of $\varphi$ in Theorem \ref{Thm:ConvBerndt} is necessary.

\begin{example}
We define an upper semi-continuous function $\varphi$ on $\Omega:=\BB^{2m}\times\BB^{2n} \subset \CC_\tau^m\times\CC_z^n$ by setting $\varphi(\tau,0)=|\tau|$ and $\varphi(\tau,z)=0$ for $z\neq0$. For any psh function $\psi$ on $\Omega$, Berndtsson's theorem implies that $\log B_{\Omega_\tau}(z;e^{-\varphi_\tau-\psi_\tau}) \equiv \log B_{\Omega_\tau}(z;e^{-\psi_\tau})$ is also psh on $\Omega$. However, $\varphi$ itself is not psh. Indeed, if we assume only that $\varphi$ is upper semi-continuous, we can still show that $\varphi$ equals a psh function almost everywhere (see Remark \ref{Rmk:aeConvBerndt}).
\end{example}

Moreover, Theorem \ref{Thm:ConvBerndt} is false without the twisted factor $e^{-\psi}$.

\begin{example} \label{Ex:ConvBerndt1}
Let $u(\tau)$ be a psh function on $U\subset\CC^m$ and $v(z)$ an upper semi-continuous function on $D\subset\CC^n$. We consider the product domain $\Omega:=U\times D$ and the function $\varphi(\tau,z):=u(\tau)+v(z)$. Clearly,
$$ \log B_{\Omega_\tau}(z;e^{-\varphi_\tau}) = u(\tau) + \log B_D(z;e^{-v}) $$
is a psh function on $\Omega$, but $\varphi$ is not necessarily psh, and $\Omega$ is not necessarily pseudoconvex.
\end{example}

In the following, we provide a more essential counterexample. Specifically, we construct a non-psh function $\varphi$ on $\CC^2$ such that, for any $\CC$-linear fibration $\CC^2\to\CC$, the associated fiberwise Bergman kernel is log-psh.

The idea is similar to that of Example \ref{Ex:ConvPrekopa3}: the function $\varphi$ we constructed is a slight perturbation of a radially symmetric strictly psh function $\varphi_0$. If $\varphi$ remains radially symmetric, then fibrations in any direction are equivalent. To simplify the computation, we also choose $\varphi$ such that the fiberwise Bergman spaces have dimension one. Stating with $\varphi_0(z,w) = \frac{3}{2}\log(1+|z|^2+|w|^2)$ on $\CC^2$, we have the following counterexample.

\begin{example}\label{Ex:ConvBerndt2}
Consider the radially symmetric function
\begin{equation}
    \varphi(z,w) := \frac{3}{2}\log\big( 1 + \big| |z|^2+|w|^2-\eps^2 \big| \big)
\end{equation}
on $\CC^2$, where $\eps\in(0,1)$ is a constant. Clearly, $\varphi$ is not psh. However, the fiberwise Bergman kernel $B_\CC(w;e^{-\varphi_z})$ is log-psh on $\CC^2$, where $\varphi_z:=\varphi(z,\cdot)$.
\end{example}

\begin{proof}
Let $z\in\CC$ be fixed for a moment. Clearly, every holomorphic function $f\in A^2(\CC;e^{-\varphi_z})$ admits a Taylor expansion $f(w)=\sum_{k=0}^\infty c_kw^k$ with compact convergence on $\CC$. Since $\varphi_z(w)$ depends only on $|w|$, it is clear that
$$ \int_\CC |f|^2e^{-\varphi_z} d\lambda = \sum_{k=0}^\infty |c_k|^2 \int_\CC |w^k|^2 e^{-\varphi_z(w)} d\lambda_w. $$
Notice that,
$$ \int_\CC |w^k|^2 e^{-\varphi_z(w)} d\lambda_w = 2\pi \int_1^\infty \frac{r^{2k+1}dr}{(1-\eps^2+|z|^2+r^2)^{3/2}} + O(1). $$
The above integral is convergent if and only if $k=0$. Therefore, the Bergman space $A^2(\CC;e^{-\varphi_z})$ contains only constant functions, and then
$$ B_\CC(w;e^{-\varphi_z}) = \frac{1}{\int_\CC e^{-\varphi_z}d\lambda} $$
is independent of $w\in\CC$. For convenience, we set $\Phi(z):=\log B_\CC(w;e^{-\varphi_z})$.

When $|z|\geq\eps$, we have
\begin{align*}
\int_\CC e^{-\varphi(z,w)} d\lambda_w & = 2\pi \int_0^\infty \frac{rdr}{(1-\eps^2+|z|^2+r^2)^{3/2}} \\
& = \frac{-2\pi}{\sqrt{1-\eps^2+|z|^2+r^2}} \big|_{r=0}^\infty = \frac{2\pi}{\sqrt{1-\eps^2+|z|^2}}.
\end{align*}
By a similar computation, for $|z|<\eps$, we have
\begin{align*}
	&~\int_\CC e^{-\varphi(z,w)} d\lambda_w = 2\pi \int_0^\infty \frac{rdr}{(1+||z|^2+r^2-\eps^2|)^{3/2}} = 4\pi - \frac{2\pi}{\sqrt{1+\eps^2-|z|^2}}.
\end{align*}
Since $\log(1+||z|^2+r^2-\eps^2|) \geq \log(1-\eps^2+|z|^2+r^2)$, it is clear that
$$ 4\pi - \frac{2\pi}{\sqrt{1+\eps^2-|z|^2}} \leq \frac{2\pi}{\sqrt{1-\eps^2+|z|^2}}. $$

In summary,
\begin{equation*}
	\Phi(z) = \begin{cases}
		\log\sqrt{1-\eps^2+|z|^2} - \log(2\pi), & |z|\geq\eps \\
		-\log\left(2-1/\sqrt{1+\eps^2-|z|^2}\right) - \log(2\pi), & |z|<\eps
	\end{cases}.
\end{equation*}
For convenience, we set
\begin{equation*}
	u(z):= \log\sqrt{1-\eps^2+|z|^2} \quad\text{and}\quad v(z):=-\log\left(2-1/\sqrt{1+\eps^2-|z|^2}\right).
\end{equation*}
Clearly, $u$ is psh on $\CC$. By a direct computation,
\begin{equation*}
	\frac{\pd^2v}{\pd z\pd\bar{z}} = \frac{(2+2\eps^2+|z|^2)\sqrt{1+\eps^2-|z|^2}-(1+\eps^2)} {2(1+\eps^2-|z|^2)^2(2\sqrt{1+\eps^2-|z|^2}-1)^2} > 0
\end{equation*}
in a neighbourhood of $\{|z|\leq\eps\}$. Hence, $\Phi$ is psh on $\CC\setminus\{|z|=\eps\}$. Given $z_0\in\CC$ with $|z_0|=\eps$, since $v\geq u$ on $\{|z|<\eps\}$ and $v=u=0$ on $\{|z|=\eps\}$, by checking the mean-value inequality for any circle centered at $z_0$, we see that $\Phi$ is also psh at $z_0$. Therefore, $\Phi$ is a psh function on $\CC$. This completes the proof.
\end{proof}

\section{The necessity of pseudoconvexity in Berndtsson's theorem} \label{Sec:PSC}

In this section, we prove Theorem \ref{Thm:CharPsc}, which concerns the necessity of pseudoconvexity in Berndtsson's theorem. Recall that a domain $\Omega\subset\CC^N$ is pseudoconvex if and only if $-\log d_\Omega$ is psh on $\Omega$, where
\begin{equation}
    d_\Omega(z) := \inf_{w\notin\Omega} |z-w|, \quad z\in\Omega,
\end{equation}
is the Euclidean distance to the boundary of $\Omega$. Therefore, to prove Theorem \ref{Thm:CharPsc}, we need to extract information about $d_\Omega$ from the weighted Bergman kernels. The starting point is the following counterpart to Lemma \ref{Lemma2}.

\begin{lemma} \label{Lemma3}
Let $D\ni0$ be a domain in $\CC_z^n$. Given a positive constant $r$, define $\phi_k:=k\max\{\log(|z|/r),0\}$, then
$$ \varliminf_{k\to+\infty} B_D(0;e^{-\phi_k}) \geq \frac{1}{|D\cap\BB^{2n}(0;r)|}. $$
Moreover, if $\BB^{2n}(0;r)\subset D$, then
$$ B_D(0;e^{-\phi_k}) \leq \frac{1}{\sigma_{2n}r^{2n}}, \quad \forall k>2n. $$
\end{lemma}

\begin{proof}
We consider the constant function $g\equiv1$, then
$$ \int_D |g|^2e^{-\phi_k} d\lambda = |D\cap\BB^{2n}(0;r)| + \int_{D\cap\{|z|\geq r\}} e^{-\phi_k} d\lambda, $$
where $|D\cap\BB^{2n}(0;r)|$ denotes the measure of  $D\cap\BB^{2n}(0;r)$. When $k>2n$, we have
$$ \int_{D\cap\{|z|\geq r\}} e^{-\phi_k} d\lambda \leq \int_{\{|z|\geq r\}} e^{-k\log(|z|/r)} d\lambda = \frac{2n}{k-2n}\sigma_{2n} r^{2n}. $$
Therefore, $g\in A^2(D;e^{-\phi_k})$ for $k>2n$ and
$$ \lim_{k\to+\infty} \int_D |g|^2e^{-\phi_k} d\lambda = |D\cap\BB^{2n}(0;r)|. $$
As a consequence,
$$ \varliminf_{k\to+\infty} B_D(0;e^{-\phi_k}) \geq \lim_{k\to+\infty} \frac{1}{\int_D |g|^2e^{-\phi_k} d\lambda} = \frac{1}{|D\cap\BB^{2n}(0;r)|}. $$

Next, we assume that $\BB^{2n}(0;r)\subset D$. For any $f\in A^2(D;e^{-\phi_k})$, we have
$$ \int_D |f|^2e^{-\phi_k} d\lambda \geq \int_{\BB^{2n}(0;r)} |f|^2 d\lambda \geq \sigma_{2n}r^{2n}|f(0)|^2. $$
Consequently,
$$ B_D(0;e^{-\phi_k}) = \sup_{f\not\equiv0} \frac{|f(0)|^2}{\int_D|f|^2e^{-\phi_k}d\lambda} \leq \frac{1}{\sigma_{2n}r^{2n}}. $$
This completes the proof.
\end{proof}

Using Lemma \ref{Lemma3}, we can prove the following interesting result.

\begin{theorem} \label{Thm:delta-dist}
Let $\Omega$ be a domain in $\CC_\tau^{m}\times\CC_z^n$ so that $\overline{\Omega_\tau}^\circ=\Omega_\tau$ for all $\tau\in p(\Omega)$. Assume that for any non-negative psh function $\psi$ on $\Omega$, the function
$$ (\tau,z) \mapsto \log B_{\Omega_\tau}(z;e^{-\psi_\tau}) $$
is either psh or identically $-\infty$ on $\Omega$. Then $-\log\delta_\Omega$ is a psh function on $\Omega$, where
\begin{equation} \label{Eq:delta-dist}
\delta_\Omega(\tau,z) := d_{\Omega_\tau}(z) =  \inf_{w\notin\Omega_\tau} |z-w|, \quad (\tau,z)\in\Omega.
\end{equation}
\end{theorem}

Here, $\delta_\Omega$ is the distance function ``in the fiber direction". If $\delta_\Omega(\tau,z)>c>0$, then $\overline{\BB^{2n}(z;c)} \subset \Omega_\tau$, and we can find $0<\eps\ll1$ such that
$$ \overline{\BB^{2m}(\tau;\eps)} \times \overline{\BB^{2n}(z;c+\eps)} \subset \Omega. $$
It follows that $\delta_\Omega(\tau',z') > c$ for any $(\tau',z') \in \BB^{2m}(\tau;\eps)\times\BB^{2n}(z;\eps)$. Therefore, $\delta_\Omega>0$ is a \textit{lower semi-continuous} function on $\Omega$. Consequently, for any compact subset $K$ of $\Omega$, we have $\delta_\Omega(K) := \inf_{K} \delta_\Omega>0$.

\begin{proof}
To prove the plurisubharmonicity of $\Phi:=-\log\delta_\Omega$, for any affine analytic disc $(\tau_0,z_0)+(\eta,\xi)\overline{\triangle}\subset\Omega$, we need to verify the mean-value inequality
\begin{equation} \label{Eq:SubMeanPhi}
	\Phi(\tau_0,z_0) \leq \frac{1}{2\pi} \int_0^{2\pi} \Phi(\tau_0+e^{i\theta}\eta, z_0+e^{i\theta}\xi) d\theta.
\end{equation}
By a standard limit argument, it suffices to consider the case where $\eta\neq0$.

We take a continuous function $\phi(e^{i\theta})$ on $\Sp^1$ such that
$$ \phi(e^{i\theta}) > \Phi(\tau_0+e^{i\theta}\eta, z_0+e^{i\theta}\xi), \quad \forall\theta\in[0,2\pi]. $$
Given a constant $\eps>0$, by the Stone-Weierstrass theorem, there exists a holomorphic polynomial $f$ on $\CC$ such that $\phi<\operatorname{Re}f<\phi+\eps$ on $\Sp^1$.

As $\eta\neq0$, we can find a holomorphic map $\gamma:\CC^m\to\CC^n$ such that $\gamma(\tau_0+w\eta)=z_0+w\xi$, and a holomorphic function $g:\CC^m\to\CC$ such that $g(\tau_0+w\eta)=f(w)$ for all $w\in\CC$. Note that $\Re g$ is pluriharmonic. We consider the following sequence of non-negative psh functions on $\Omega$,
\begin{equation}
	\psi_k(\tau,z) := k\max\big\{\log|z-\gamma(\tau)|+\Re g(\tau),0\big\}, \quad k\geq2n+1.
\end{equation}
By assumptions,
$$ \Psi_k(\tau,z) := \log B_{\Omega_\tau}(z;e^{-\psi_{k,\tau}}) $$
are psh functions on $\Omega$. In particular,
$$ u_k(w) := \Psi_k(\tau_0+w\eta, z_0+w\xi) $$
are subharmonic functions in a neighborhood of $\overline{\triangle}$.

Given $w_*\in\pd\triangle$, we denote $(\tau_*,z_*):=(\tau_0+w_*\eta,z_0+w_*\xi)$. Since
$$ -\log\delta_\Omega(\tau_*,z_*) = \Phi(\tau_*,z_*) < \phi(w_*) < \Re f(w_*) = \Re g(\tau_*), $$
we know
$$ \BB^{2n}(z_*;e^{-\Re g(\tau_*)}) \subset \Omega_{\tau_*} \subset \CC^n. $$
According to Lemma \ref{Lemma3},
$$  B_{\Omega_{\tau_*}}(z_*;e^{-\psi_{k,\tau_*}}) \leq \frac{1}{\sigma_{2n}e^{-2n\Re g(\tau_*)}}. $$
Consequently,
$$ u_k(w_*) \leq 2n\Re f(w_*)-\log\sigma_{2n}, \quad \forall w_*\in\pd\triangle. $$
Since $u_k$ is subharmonic and $\Re f$ is harmonic, it follows that
$$ u_k(0) \leq 2n\Re f(0)-\log\sigma_{2n}, \quad \forall k\geq 2n+1. $$

We claim that $\delta_\Omega(\tau_0,z_0) \geq e^{-\Re f(0)}$. Otherwise, there exists a point $z_*\notin\Omega_{\tau_0}$ with $|z_*-z_0|<e^{-\Re f(0)}$. Since $\overline{\Omega_{\tau_0}}^\circ = \Omega_{\tau_0}$, it follows that $|\BB^{2n}(z_*;s)\setminus\Omega_{\tau_0}|>0$ for any $s>0$, and then
$$ |\Omega_{\tau_0} \cap \BB^{2n}(z_0;e^{-\Re f(0)})| < \sigma_{2n}e^{-2n\Re f(0)}. $$
According to Lemma \ref{Lemma3},
$$ \varliminf_{k\to+\infty} u_k(0) = \varliminf_{k\to+\infty} \log B_{\Omega_{\tau_0}}(z_0;e^{-\psi_{k,\tau_0}}) > 2n\Re f(0) -\log\sigma_{2n}. $$
This leads to a contradiction! Therefore, $\delta_\Omega(\tau_0,z_0) \geq e^{-\Re f(0)}$, and then
\begin{equation*}
	\Phi(\tau_0,z_0) \leq \Re f(0) = \frac{1}{2\pi} \int_0^{2\pi} \Re f(e^{i\theta}) d\theta < \frac{1}{2\pi} \int_0^{2\pi} \phi(e^{i\theta}) d\theta + \eps.
\end{equation*}
Letting $\eps\searrow0$ and $\phi(e^{i\theta}) \searrow \Phi(\tau_0+e^{i\theta}\eta, z_0+e^{i\theta}\xi)$, we complete the proof.
\end{proof}

We now show that in the case where $m=1$, the plurisubharmonicity of $-\log\delta_\Omega$ implies the plurisubharmonicity of $-\log d_\Omega$, and thus the pseudoconvexity of $\Omega$. This will complete the proof of Theorem \ref{Thm:CharPsc}. To this end, we need the following definition:

A \textit{closed analytic disc} in $\CC^N$ is a non-constant holomorphic map $\bd:\triangle\to\CC^N$ that extends continuously to $\overline{\triangle}$. For simplicity, we intentionally confuse $\bd$ with its image and refer to $\pd\bd := \bd(\pd\triangle)$ as the boundary of $\bd$.

\begin{theorem}
Let $\Omega$ be a domain in $\CC_\tau\times\CC_z^n$ and we define $\delta_\Omega(\tau,z)$ by \eqref{Eq:delta-dist}. If $-\log\delta_\Omega$ is a psh function on $\Omega$, then $-\log d_\Omega$ is also psh.
\end{theorem}

\begin{proof}
According to Hartogs, to prove the plurisubharmonicity of $-\log d_\Omega$, it is sufficient to show that $d_\Omega(\bd)=d_\Omega(\pd\bd)$ for any closed analytic disc
$$ \bd: w\in\overline{\triangle} \mapsto (f(w),g(w))\in\CC\times\CC^n $$
contained in $\Omega$ (see \cite[Theorem 3.3.5, (7) $\Rightarrow$ (1)]{KrantzBook}). Actually, we may assume that $f(w)$ is non-constant, $f$ and $g$ are holomorphic in a neighborhood of $\overline{\triangle}$.\footnote{According to Hartogs' proof, to verify the mean-value inequality for $-\log d_\Omega$ on $(\tau_0,z_0)+(\eta,\xi)\overline{\triangle}\subset\Omega$, it suffices to show that $d_\Omega(\bd)=d_\Omega(\pd\bd)$ for closed analytic discs $\bd\subset\Omega$ of the form
$$ \bd: w\in\overline{\triangle} \mapsto (\tau_0+w\eta+\alpha e^{-p(w)},z_0+w\xi+\beta e^{-p(w)}), $$
where $\alpha\in\CC$, $\beta\in\CC^n$, and $p(w)$ is a holomorphic polynomial on $\CC$. By a standard limit argument, it suffices to consider the case where $\eta\neq0$, and then $w\mapsto \tau_0+w\eta+\alpha e^{-p(w)}$ is non-constant.}

Clearly, $d_\Omega(\bd) \leq d_\Omega(\pd\bd)$. Therefore, we remain to prove the reverse inequality. Assume that the minimal distance
$$ r := d_\Omega(\bd) = \min_{w\in\overline{\triangle}} d_\Omega(f(w),g(w)) > 0 $$
is obtained at $w_0\in\overline{\triangle}$. If $w_0\in\pd\triangle$, we done. Hence, we assume that $w_0\in\triangle$. By definition, there exists a vector $(\eta,\xi)\in\CC\times\CC^n$ of length $r$ such that
$$ (f(w_0)+\eta,g(w_0)+\xi) \notin\Omega. $$
We divide the proof into two cases.

\textbf{Case 1:} $\xi\neq0$. In this case, we consider the closed analytic disc
$$ \bd_\eta: w\in\overline{\triangle} \mapsto (f(w)+\eta, g(w)) \in \CC\times\CC^n, $$
which is the horizontal shift of $\bd$ by $\eta$. Since $|\eta|<r=d_\Omega(\bd)$, we know $\bd_\eta\subset\Omega$. Since $-\log\delta_\Omega$ is psh on $\Omega$, it follows that
$$ w\mapsto-\log\delta_\Omega(f(w)+\eta,g(w)) $$
is a subharmonic function in a neighborhood of $\overline{\triangle}$. By the maximum principle of subharmonic functions, we have
$$ \delta_\Omega(\pd\bd_\eta) = \delta_\Omega(\bd_\eta). $$
As $(f(w_0)+\eta,g(w_0)+\xi) \notin\Omega$, we know $\delta_\Omega(\bd_\eta) \leq |\xi|$, and then
$$ \delta_\Omega(\pd\bd_\eta) = \delta_\Omega(\bd_\eta) \leq |\xi|. $$
In other words, there exists some $w_*\in\pd\triangle$ such that
$$ \delta_\Omega(f(w_*)+\eta, g(w_*)) \leq |\xi|. $$
Then we can find a $\zeta\in\CC^n$ such that $|\zeta|\leq|\xi|$ and $(f(w_*)+\eta, g(w_*)+\zeta)\notin\Omega$. Consequently,
$$ d_\Omega(\pd\bd) \leq d_\Omega(f(w_*), g(w_*)) \leq \sqrt{|\eta|^2+|\zeta|^2} \leq r = d_\Omega(\bd). $$

\textbf{Case 2:} $\xi=0$. In this case, $|\eta|=r$ and
$$ (\hat{\tau},\hat{z}) := (f(w_0)+\eta, g(w_0)) \notin\Omega. $$
Recall that, $f:\overline{\triangle}\to\CC$ is non-constant and holomorphic. By the open mapping theorem, for any $0<\eps\ll1$, there exists a point $w_\eps\in\triangle$ closed to $w_0$ such that $f(w_\eps)=f(w_0)+\eps\eta$. Clearly, $w_\eps\to w_0$ as $\eps\to0$. We consider the closed analytic disc
$$ \bd_\eps: w\in\overline{\triangle} \mapsto (f(w)+(1-\eps)\eta,g(w))\in\CC\times\CC^n, $$
which is the horizontal shift of $\bd$ by $(1-\eps)\eta$. Since $|(1-\eps)\eta|<r=d_\Omega(\bd)$, we know $\bd_\eps\subset\Omega$. By similar arguments as above, we have $\delta_\Omega(\bd_\eps) = \delta_\Omega(\pd\bd_\eps)$. Since
$$ f(w_\eps)+(1-\eps)\eta = f(w_0)+\eta = \hat{\tau} $$
and $(\hat{\tau},\hat{z})\notin\Omega$, it is clear that
$$ \delta_\Omega(\pd\bd_\eps) = \delta_\Omega(\bd_\eps) \leq |g(w_\eps)-\hat{z}| = |g(w_\eps)-g(w_0)|. $$
As a consequence,
$$ d_\Omega(\pd\bd) \leq \sqrt{|(1-\eps)\eta|^2+|g(w_\eps)-g(w_0)|^2}. $$
By letting $\eps\to0$, we conclude that $d_\Omega(\pd\bd) \leq |\eta|=r=d_\Omega(\bd)$.

In summary, $d_\Omega(\bd)=d_\Omega(\pd\bd)$ in both case. This completes the proof.
\end{proof}

\end{document}